\documentclass[12pt,reqno]{amsart}
\usepackage{geometry}
\usepackage[latin1]{inputenc}
\usepackage[italian,english]{babel}
\usepackage{amsmath, amsfonts, amsthm, xcolor}
\usepackage{paralist}
\numberwithin{equation}{section}
\usepackage{pgfplots}
\pgfplotsset{compat=1.15}
\usepackage{mathrsfs}
\usetikzlibrary{arrows}
\usepackage{mathtools, mathabx, verbatim}
\usepackage[colorlinks=true, allcolors=blue]
{hyperref}
\usepackage{hyperref}
\usepackage{cleveref}

\DeclareMathOperator{\spt}{spt}
\newcommand{\R}{\mathbb{R}}
\newcommand{\de}{\partial}
 \newcommand{\ds}{\displaystyle}

\newtheorem{theorem}{Theorem}[section]
\newtheorem{lemma}[theorem]{Lemma}

\newtheorem{proposition}[theorem]{Proposition}
\newtheorem{remark}[theorem]{Remark}
\newtheorem{definition}[theorem]{Definition}
\definecolor{ududff}{rgb}{0.30196078431372547,0.30196078431372547,1}
\definecolor{ccqqqq}{rgb}{0.8,0,0}

\title[The first Robin-Dirichlet eigenvalue]{A sharp bound for the first Robin-Dirichlet eigenvalue}
\author{Nunzia Gavitone}
\address{Dipartimento di Matematica e Applicazioni ``R. Caccioppoli'', Universit\`a degli studi di Napoli Federico II, Via Cintia, Complesso Universitario Monte S. Angelo, 80126 Napoli, Italy.}
\email{nunzia.gavitone@unina.it (corresponding author)}

\author{Gianpaolo Piscitelli}
\address{Dipartimento di Scienze Economiche, Giuridiche, Informatiche e Motorie, Universit\`a degli Studi di Napoli Parthenope, Via Guglielmo Pepe, Rione Gescal, 80035 Nola (NA), Italy.}
\email{gianpaolo.piscitelli@uniparthenope.it}

\begin{document}

\date{}
\begin{abstract}
In this paper, we study the first eigenvalue  of the Laplacian on doubly connected domains when Robin and Dirichlet conditions are imposed  on the outer and the inner part of the boundary, respectively. We provide that the spherical shell reaches the maximum of the first eigenvalue of this problem among a suitable class of domains when the measure, the outer perimeter and inner $(n-1)^{\text{th}}$ quermassintegral are fixed. 

\vspace{.3cm}
\vspace{.3cm}

\centering {\it Communicated by Irena Lasiecka.}

\end{abstract}
\maketitle

\section{Introduction}
In the recent literature, the study of shape optimization eigenvalue problems for the Laplacian in doubly connected domains, has been widely developed. In this kind of problems, different  conditions can be imposed on the two components of  the boundary and under suitable geometric constraints,  it is worth asking whether the spherical shell is the optimal shape.  
The study of the shape optimization problems of the first Laplacian eigenvalue in doubly connected domains,  strictly  depends on the conditions imposed on the two parts of the boundary.    More precisely, let  $\Omega_0, \Theta\subset\R^n$, $n\ge 2$, be two open, connected,  bounded sets such that $\Omega_0$ has Lipschitz boundary and $\Theta\Subset\Omega_0$. Then we consider the annular set $\Omega=\Omega_0\setminus\overline{\Theta}$.

In  this paper we study the first  Robin-Dirichlet eigenvalue of the Laplacian, that is
\begin{equation}
    \label{min_pb_intro}
\lambda(\beta,\Omega)=\min_{\substack{\psi\in H^1_{\partial\Theta}(\Omega)\\ \psi \nequiv 0} } \frac{\displaystyle\int_\Omega|\nabla \psi|^2dx+\beta\int_{\partial\Omega_0}\psi^2d\mathcal H^{n-1}}{\displaystyle\int_\Omega \psi^2dx},
\end{equation}
where $\beta>0$ is a positive parameter and  $H^1_{\partial\Theta}(\Omega)$ denotes the set of the Sobolev functions in $\Omega$ which vanish on the boundary of $\Theta$ (see \Cref{sec_3_problem} for the precise definition). Any minimizer $ u \in H^1_{\partial\Theta}(\Omega) $ of \eqref{min_pb_intro} is a solution to the following problem (see \Cref{sec_3_problem} for the details)
\begin{equation}
\label{eig_pb_intro}
\begin{cases}
-\Delta u =\lambda(\beta,\Omega) \,u & \text{ in } \Omega\\
\partial_\nu u+\beta u=0 & \text{ on }\partial\Omega_0
\\
u=0 & \text{ on }\partial\Theta,
    \end{cases}
\end{equation}
where $\nu$ is the outer unit normal to $\partial\Omega_0$ and $\partial_\nu u$ denotes the normal derivative of $u$. 

We observe that when $\beta=0$ and $\beta=+\infty$, \eqref{min_pb_intro} reduces to the so-called first Neumann-Dirichlet  and Dirichlet-Dirichlet Laplacian eigenvalue, respectively. 

In the case $\beta=0$, the first Neumann-Dirichlet Laplacian  eigenvalue $\lambda_{ND}(\Omega)$ was studied  for instance in \cite{hersch1962contribution,anoop2023reverse,dellapietra2020optimal,anoop2021shape}. 
When the hole $\Theta$ is convex, in  \cite{anoop2023reverse,dellapietra2020optimal} the authors proved that the spherical shell $A_{R_1,R_2}=B_{R_2} \setminus \overline{B_{R_1}} $ maximizes  $\lambda_{ND}(\Omega)$ among all the doubly connected domains $\Omega$ with given volume and such that  the hole $\Theta$  has fixed $(n-1)^{\text{th}}$ quermassintegral $W_{n-1}$ (see \Cref{sec_2_not} for the precise definition). This result in the planar case was proved in \cite{hersch1962contribution} and we point out that in this case the constraint on the hole $\Theta$ corresponds to fix its perimeter and the maximizer is the annulus.

When $\beta =+\infty$, $\lambda(\beta,\Omega)$ reduces to the first Dirichlet-Dirichlet eigenvalue $\lambda_{DD}(\Omega)$ studied for instance in \cite{ii2001placement,kesavan2003two,el2008extremal,georgiev2018maximizing,henrot2019optimizing, chorwadwala2020place, paoli2020sharp,chorwadwala2022optimal,cito2023stability}.
 In particular, in \cite{hersch1963method}, the author proved that in the planar case,  the annulus maximizes $\lambda_{DD}(\Omega)$  among  the doubly connected plane domains $\Omega$ of  area $A$ and such that  $4\pi A=L_{\text{out}}^2-L_{\text{int}}^2$ where $L_{\text{out}}$ and $L_{\text{int}}$ are the lengths of the outer and inner boundary respectively. Up to our knowledge, in any dimension this result has been proved in \cite{kesavan2003two} only for eccentric spherical shells, that is when $\Omega_0$ and $\Theta$ are two Euclidean balls not concentric.

The purpose of this paper is to find optimal upper bound for $\lambda(\beta,\Omega)$ for any $\beta>0$ under suitable geometrical constraints on $\Omega$, in the spirit of the quoted results for $\lambda_{DD}(\Omega)$ and $\lambda_{ND}(\Omega)$.  More precisely, let us consider the following family of admissible doubly connected sets.
\begin{definition}
 Let  $\mathcal S$ be the class of doubly connected domains $\Omega=\Omega_0\setminus\overline{\Theta}$ where $\Omega_0, \Theta\subset\R^n$, $n\ge 2$, are two open, bounded, convex sets and
 \begin{equation}
  \displaystyle \frac{|\Omega|}{\omega_n}= \displaystyle \left(\frac{P(\Omega_0)}{n\omega_ n}\right)^\frac{n}{n-1}-\displaystyle \left(\frac{(n-1)W_{n-1}(\Theta)}{\omega_ n}\right)^n.
  \end{equation}
\end{definition}

Then the main result of this paper is the following upper bound for $\lambda(\beta,\Omega)$.

\begin{theorem}
\label{thm_isop}
Let  be $\beta>0$ and let be $\Omega\in \mathcal S$. Then
\begin{equation}
\label{main_th}
\lambda(\beta,\Omega)\leq \lambda(\beta, A_{R_1,R_2}) 
\end{equation}
where $A_{R_1,R_2}=B_{R_2} \setminus \overline{B_{R_1}}$ is the spherical shell such that,  $P(B_{R_2})=P(\Omega_0)$, $W_{n-1}(B_{R_1})=W_{n-1}(\Theta)$ and $|A_{R_1,R_2}|=|\Omega|$.
\end{theorem}

To prove \Cref{thm_isop} we use the so-called inner parallel method (introduced by Makai \cite{makai1959bounds,makai1959principal} and P\'olya \cite{polya1960two} and then refined by Payne and Weinberger in \cite{payne1961some} basing on \cite{nagy1959parallelmengen}). In our paper, we use a generalized version of this tool for the higher dimension, the so-called web-functions method (see for instance \cite{gazzola1999existence,crasta2002sharp,brandolini2010upper,bucur2019sharp,bucur2022shape,amato2023sharp} and \cite{mana} for the anisotropic case). More precisely, we use a suitable test function which is defined by means of the distance function from the boundary.

Moreover, we point out that in order to prove \Cref{thm_isop} we need to study the problem \eqref{eig_pb_intro} in the radial case, that is when $\Omega=A_{R_1,R_2}$ is the spherical shell. We prove that in this case the first eigenfunctions $v$ are radially symmetric that is there exists a function $\phi(r)$ such that $v(x)=\phi(|x|)$, and satisfying the following problem (see \Cref{sec_3_problem} for the details)
\[
\begin{cases}
- \frac{1}{r^{n-1}}\left(\phi'(r)r^{n-1}\right)'=\lambda(\beta, A_{R_1,R_2}) \phi(r) & \text{ for } r\in(R_1,R_2)\\
\phi ' (R_2)+\beta \phi(R_2)=0 \\
\phi(R_1)=0.
\end{cases}
\]

In particular, we prove that $\phi(r)$, for any $R_1<r<R_2$, is not globally monotone and has  a unique critical point $\bar R \in(R_1,R_2)$ (see \Cref{prop_existence}). This result allows us to study the shape optimization problem related to $\lambda(\beta,\Omega)$, by merging  both the Robin-Neumann and the Neumann-Dirichlet eigenvalue problems by using the inner parallel method.

Furthermore, since we show that 
\[
\lim_{\beta \to + \infty}\lambda(\beta,\Omega)=\lambda_{DD}(\Omega),
\]
then \Cref{thm_isop} also gives 
\begin{equation}
\label{iso_DD}
\lambda_{DD}(\Omega)\leq \lambda_{DD}(A_{R_1,R_2}).
\end{equation}
The relationship \eqref{iso_DD} extends the Hersch's result in \cite{hersch1963method} for $\lambda_{DD}(\Omega)$ to any dimension and to a larger class of doubly connected domains than  the ones considered  in \cite{kesavan2003two}. Indeed eccentric shells belong in the class of admissible sets $\mathcal S$. 
 
The plan of the paper is the following. In \Cref{sec_2_not}, we provide the notation and the preliminaries; in \Cref{sec_3_problem}, we investigate the foundation of the problem and we study the radial case; in \Cref{sec_4_shape}, we compute the first order shape derivative of $\lambda(\beta,\Omega)$ in order to prove that the spherical shell is a critical set; finally, in \Cref{sec_5_isop}, we prove \Cref{thm_isop}.

\section{Notation and Preliminaries} 
\label{sec_2_not}
Throughout this paper, we denote by $B_r(x_0)$ and $B_r$ the balls in $\mathbb{R}^n$ of radius $r>0$, centered at  $x_0\in\mathbb{R}^n$ and at the origin, respectively. Moreover by $B$, $\mathbb{S}^{n-1}$ and $\omega_n$ we will denote respectively the unit ball of $\mathbb R^{n}$, its boundary and its volume. 
Let $0<r<R$, then the spherical shell with radii $r$ and $R$ centered at the origin is denoted as
\begin{equation*}
A_{r,R}=\{x \in \mathbb R^{n} \colon r<|x|<R\}.
\end{equation*}
Moreover, the $(n-1)$-dimensional Hausdorff measure in $\mathbb{R}^n$ will be denoted by $\mathcal H^{n-1}$ and, the Euclidean scalar product in $\mathbb{R}^n$ by $\langle\cdot,\cdot\rangle$.

Let $D\subseteq\mathbb{R}^n$ be an open, bounded set. For any measurable set $E\subseteq\R^{n}$, the perimeter of $E$ in $D$ is (see for instance \cite{maggi2012sets})
\begin{equation*}
P(E;D)=\sup\left\{  \int_E {\rm div} \varphi\:dx :\;\varphi\in C^{\infty}_c(D;\mathbb{R}^n),\;||\varphi||_{\infty}\leq 1 \right\}.
\end{equation*}
The perimeter of $E$ in $\mathbb{R}^n$ will be denoted by $P(E)$ and, if $P(E)<+\infty$, we say that $E$ is a set of finite perimeter. Moreover, if $E$ has Lipschitz boundary, it holds
\[
P(E)=\mathcal H^{n-1}(\partial E).
\]

The Lebesgue measure of a measurable set $E \subset \mathbb R^{n}$ will be denoted by $|E|$.
Moreover, we recall that the inradius of $E\subset \mathbb R^{n}$ is defined as follows 
\begin{equation}\label{inradius_def}
\rho( E)=\sup_{x\in E}\; \inf_{y\in\partial E} |x-y|,
\end{equation} 
while the diameter of $E$ is
\[
\text{diam} (E)=\sup_{x,y \in E} |x-y|.
\]
Let us observe that, if $E\subset \R^n$ is a bounded, convex set, it holds (see for instance \cite{della2018sharp,schneider2014convex})
\begin{equation}
\label{pmi}
\ds\frac{\rho(E)}{n}\le \ds \frac{|E|}{P(E)} \le \rho(E)
\end{equation}

In what follows we recall some properties of convex bodies, which will be useful in the following. For more details on these topics we refer for instance to \cite{kesavan2006symmetrization,schneider2014convex}.

Let $K$ be a bounded convex open set, the outer parallel body of $K$ at distance $\delta>0$ is the following Minkowski sum
\[ 
K+\delta B=\{ x+\delta y\in\mathbb{R}^n\;|\; x\in K,\;y\in B \}.
\]
The Steiner formulas assert that
\begin{equation*}\label{general_steiner}
|K+\delta B|=\sum_{i=0}^{n}\binom{n}{i} W_i(K)\delta^i,
\end{equation*}
and
\begin{equation*}\label{general_steiner_per}
P(K+\delta B)=n\sum_{i=0}^{n-1}\binom{n-1}{i} W_{i+1}(K)\delta^{i}. 
\end{equation*}
The coefficients $W_{i}(K)$ are known as the quermassintegrals of $K$. In particular, it holds that
\[
W_{0}(K)=\left|K\right|,\quad nW_{1}(K)= P(K), \quad W_{n}(K)=\omega_{n}.
\]

Furthermore, the Aleksandrov-Fenchel inequalities state that
\begin{equation}
  \label{afineq}
\left( \frac{W_j(K)}{\omega_n} \right)^{\frac{1}{n-j}} \ge \left(
  \frac{W_i(K)}{\omega_n} \right)^{\frac{1}{n-i}}, \quad 0\le i < j
\le n-1,
\end{equation}
where the inequality is replaced by an equality if and only if $K$ is a ball. 

When $i=0$ and $j=1$, \eqref{afineq} reduces to the classical isoperimetric inequality:
\begin{equation}
    \label{aleksandrov-fenchel}
P(K) \ge n \omega_n^{\frac 1 n} |K|^{1-\frac 1 n}.
\end{equation}


It  the sequel it will be  useful \eqref{afineq} in the case  $i=1$ and $j=2$, that is
\begin{equation}\label{iso_W_2}
W_2(K)\geq n^{-\frac{n-2}{n-1}}\omega_n^{\frac{1}{n-1}}P(K)^{\frac{n-2}{n-1}}.
\end{equation}

Let us denote by $K^{*}$ a ball such that $W_{n-1}(K)=W_{n-1}(K^{*})$. Then, by Aleksandrov-Fenchel inequalities \eqref{afineq}, we have
\[
\left(\frac{W_{i}(K)}{\omega_{n}}\right)^{\frac{1}{n-i}}\le
\frac{W_{n-1}(K)}{\omega_{n}} =
\left(\frac{W_{i}(K^{*})}{\omega_{n}}\right)^{\frac{1}{n-i}},
\]
for $0\le i < n-1$, and hence 
\begin{equation*}
\label{afapp}
W_{i}(K)\le W_{i}(K^{*})\quad 0\le i<n-1.
\end{equation*}


Finally we recall a useful result whose proof can be found for instance in \cite{brandolini2010upper} and \cite{bucur2019sharp}.
\begin{lemma}
Let $\Omega_0$ be a bounded, open, convex set of $\R^n$  and let $f:[0,+\infty)\to[0,+\infty) $  be a non decreasing $C^1$ function
. Let $u(x):=f(d_o(x))$, where $d_o$ is the distance function from the boundary $\de \Omega_0$ and
\begin{equation*}
	\begin{split}	
\mathsf  E_{t}&:=\{x\in \Omega_0\;:\; u(x)>t\}.\\
	\end{split}
	\end{equation*}
Then, it holds
\begin{equation}\label{derivata_composta_super1}
		-\dfrac{d}{dt}P(\mathsf  E_{t})\geq n (n-1)\dfrac{W_2(\mathsf  E_{t})}{|D u|_{u=t}} \quad\text{ for a.e. }t \in (0,\rho(\Omega_0)).
	\end{equation}
\end{lemma}

\section{The Robin-Dirichlet eigenvalue problem}
\label{sec_3_problem}
Let  $\Omega_0, \Theta\subset\R^n$, $n\ge 2$, be two open, bounded, connected  sets  such that $\Omega_0$ has Lipschitz boundary and  $\Theta\Subset\Omega_0$. Let us consider the doubly connected set $\Omega=\Omega_0\setminus\overline{\Theta}$. In this section we study the following eigenvalue problem
\begin{equation}
\label{eig_pb}
\begin{cases}
-\Delta u =\lambda \,u & \text{ in } \Omega\\
\partial_\nu u+\beta u=0 & \text{ on }\partial\Omega_0
\\
u=0 & \text{ on }\partial\Theta,
    \end{cases}
\end{equation}
where $\beta>0$ is a positive parameter, $\nu$ is the outer unit normal to $\partial\Omega_0$ and $\partial_\nu u$ denotes the normal derivative of $u$.
Let us consider the set $H^1_{\partial\Theta}(\Omega)$ of the Sobolev functions in $\Omega$ vanishing on $\overline{\Theta}$, that is the closure in $H^1(\Omega)$ of
\[
C^\infty_{\partial\Theta} (\Omega):=\{ u|_{\Omega}  \ | \ u \in C_0^\infty (\R^n),\ \spt (u)\cap \partial\Theta=\emptyset \}.  
\]
We first give the definition of eigenvalue and eigenfunction to the problem \eqref{eig_pb}.
\begin{definition}\label{weak_sol_def}
The real number $\lambda$ is an eigenvalue of \eqref{eig_pb} if and only if there exists a function $z\in H^{1}_{\partial\Theta}(\Omega)$, not identically zero, such that 
\begin{equation}
    \label{weak_sol}
\int_{\Omega}\langle\nabla z,\nabla\varphi\rangle \;dx+ \beta \int_{\partial\Omega_0} z\varphi\;d\mathcal{H}^{n-1}=\lambda\int_{\Omega}z \varphi \;dx
\end{equation}
	for every $\varphi\in H^{1}_{\partial\Theta}(\Omega)$. The function $u$ is called an eigenfunction corresponding to $\lambda$.
\end{definition}

\subsection{The first eigenvalue and basic properties}
In this section, we study the main properties of the first eigenvalue to the problem \eqref{eig_pb} and of the corresponding eigenfunctions. Let us consider the following quantity
 \begin{equation}
    \label{min_pb}
\lambda(\beta,\Omega)=\inf_{\substack{\psi\in H^1_{\partial\Theta}(\Omega)\\ \psi \nequiv 0} }\frac{\displaystyle\int_\Omega|\nabla \psi|^2dx+\beta\int_{\partial\Omega_0}\psi^2d\mathcal H^{n-1}}{\displaystyle\int_\Omega \psi^2dx}.
\end{equation}
By using standard argument of the Calculus of Variations, it is possible to prove that $\lambda(\beta,\Omega)$ is the first eigenvalue to the problem \eqref{eig_pb} and we get the following result.
\begin{proposition}
\label{prop_existence}
Let be $\beta>0$ and  $\Omega=\Omega_0\setminus\overline{\Theta}$, where  $\Omega_0, \Theta\subset\R^n$, $n\ge 2$, are two open, bounded, connected  sets  such that $\Omega_0$ has Lipschitz boundary and  $\Theta\Subset\Omega_0$. Then $\lambda(\beta,\Omega)$ is positive and it is actually a minimum. Moreover, any minimizer $u\in H^{1}_{\partial\Theta}(\Omega)$ of \eqref{min_pb} has constant sign  and is a weak solution to the problem \eqref{eig_pb}. \end{proposition}
Furthermore, we are able to prove that the first eigenvalue is simple.
\begin{proposition}
\label{prop_simpl}
Let be $\beta>0$ and $\Omega=\Omega_0\setminus\overline{\Theta}$, where  $\Omega_0, \Theta\subset\R^n$, $n\ge 2$, are two open, bounded, connected  sets  such that $\Omega_0$ has Lipschitz boundary and  $\Theta\Subset\Omega_0$. Then $\lambda(\beta,\Omega)$ is simple, that is all the associated eigenfunctions are scalar multiple of each other.
\end{proposition}
\begin{proof}
Let $u,v$ be two minimizers of (\ref{min_pb}). By \Cref{prop_existence} we can assume that $u$ is positive in $\Omega$. Then we have
\begin{equation*}
    \int_{\Omega} u(x) \,dx> 0.
\end{equation*}
Therefore, we can find a real constant $\chi$ such that
\begin{equation} \label{nullmean}
    \int_{\Omega} (u(x)-\chi v(x)) \,dx = 0.
\end{equation}
Since $u-\chi v$  is still a solution of the problem (\ref{min_pb}), then we necessarily have that $u \equiv \chi v $ in $\Omega$ and the simplicity of $\lambda(\beta,\Omega)$ follows.
\end{proof}
Finally, we prove that all eigenfunctions with constant sign are the first ones.
\begin{proposition}
Let be $\beta>0$ and $\Omega=\Omega_0\setminus\overline{\Theta}$, where  $\Omega_0, \Theta\subset\R^n$, $n\ge 2$, are two open, bounded, connected  sets  such that $\Omega_0$ has Lipschitz boundary and  $\Theta\Subset\Omega_0$. Then, any positive function $z\in H^1_{\partial\Theta}(\Omega)$ that satisfies \eqref{eig_pb}, in the sense of Definition \ref{weak_sol_def},
is a first eigenfunction.
\end{proposition}
\begin{proof}
Let $u$ be a positive eigenfunction corresponding to $\lambda(\beta,\Omega)$; we have 
\begin{equation}\label{testfuncu}
	\int_{\Omega} |\nabla u|^2 \;dx+\beta\int_{\partial \Omega_0}u^2 \,d\mathcal{H}^{n-1}=\lambda(\beta,\Omega)\int_{\Omega}u^2 \;dx.
\end{equation}
For any $\varepsilon>0$, choosing $\varphi=\ds\frac{u^2}{z+\varepsilon}$ in \eqref{weak_sol}, we get
\begin{equation}\label{testfuncv}
\begin{split}
\int_{\Omega} \left[\frac{2u}{z+\varepsilon}\left(\nabla z\cdot\nabla u\right)-\frac{u^2}{(z+\varepsilon)^2}|\nabla z|^2\right]\,dx +\beta\int_{\partial \Omega_0}\frac{zu^2}{z+\varepsilon} \,d\mathcal{H}^{n-1}   & =\lambda\int_{\Omega}\frac{zu^2}{z+\varepsilon} \;dx.
\end{split}\end{equation}
Subtracting \eqref{testfuncu} by \eqref{testfuncv}, being $\ds\frac{z}{z+\varepsilon}<1$, we get
\begin{equation*}
    0\le \int_{\Omega}\left| \nabla u -\frac{u}{z+\varepsilon}\nabla z \right|^2\,dx \le \int_{\Omega}\left(\lambda(\beta,\Omega)-\lambda\frac{z}{z+\varepsilon}\right)u^2\;dx.
\end{equation*}
Passing to the limit as $\varepsilon\to 0^+$, we get 
\begin{equation*}
\left(\lambda(\beta,\Omega)-\lambda\right)
 \int_{\Omega}u^2\;dx\ge 0.
\end{equation*}
Since $\lambda(\beta,\Omega)$ is the smallest eigenvalue, this implies that $\lambda(\beta,\Omega)=\lambda$. Therefore, by the simplicity of the first eigenvalue, $z=u$ up to a multiplicative constant.
\end{proof}

\subsection{The behavior with respect to $\beta$}
In this section, we study the  behavior of the first eigenvalue \eqref{min_pb} with respect to the Robin parameter.
\begin{proposition}
Let be $\beta>0$ and $\Omega=\Omega_0\setminus\overline{\Theta}$, where  $\Omega_0, \Theta\subset\R^n$, $n\ge 2$, are two open, bounded, connected  sets  such that $\Omega_0$ has Lipschitz boundary and  $\Theta\Subset\Omega_0$. Then, $ \lambda(\beta,\Omega)$ is non-decreasing, continuous and differentiable with respect to $\beta$ and it holds
\[
\frac{d}{d\beta} \lambda(\beta,\Omega)=\displaystyle \frac{ \displaystyle\int_{\partial\Omega_0}u^2\;d\mathcal{H}^{N-1}}{ \displaystyle\int_{\Omega}u^2\;dx},
\]
where $u$ is a minimizer of \eqref{min_pb}.
\end{proposition}
\begin{proof}
The monotonicity with respect to $\beta$ is obvious. For the proof of the  differentiability refer to \cite[Prop. 2.3]{dellapietra2020optimal}.
\end{proof}%
\begin{remark}
The previous result gives
\begin{equation}
    \label{asy_ND}
\lim_{\beta \to 0^+}\lambda(\beta,\Omega)=\lambda_{{ND}}(\Omega),
\end{equation}
where $\lambda_{{ND}}(\Omega)$  is the first Neumann-Dirichlet eigenvalue of the Laplacian (see for instance \cite{hersch1962contribution,anoop2023reverse,dellapietra2020optimal}).
\end{remark}
In what follows, we prove some estimates on $\lambda(\beta,\Omega)$ in the spirit of the ones contained in \cite{gavitone2023steklov} for the first Steklov-Robin Laplacian eigenvalue (see also \cite{di2022two} and  \cite{kuttler1969inequality}). 
\begin{theorem}\label{mainstime}
Let be $\beta>0$ and $\Omega=\Omega_0\setminus\overline{\Theta}$, where  $\Omega_0, \Theta\subset\R^n$, $n\ge 2$, are two open, bounded, connected  sets  such that $\Omega_0$ has Lipschitz boundary and  $\Theta\Subset\Omega_0$.  
Then, the following estimates hold
	\label{main}
	\begin{equation}\label{Kutt_40i}
	\frac{1}{\lambda_{{DD}}(\Omega)} \le\frac{1}{\lambda(\beta,\Omega)}\leq\frac{1}{\lambda_{{DD}}(\Omega)}+\frac{1}{\beta q(\Omega) },
	\end{equation}
	where $\lambda_{{DD}}(\Omega)$ is the first Dirichlet-Dirichlet eigenvalue of $\Omega$  and  $q(\Omega)$ is
\begin{equation}
\label{q}
    q(\Omega) = \inf_{\displaystyle\substack{\Delta w=0\\ w\in H^1(\Omega)\setminus \{0\} }}\frac{\displaystyle\int_{\partial \Omega_0}w^2\,d\mathcal{H}^{n-1}}{\displaystyle\int_{\Omega}w^2\,dx}.
\end{equation}
\end{theorem}
\begin{proof}
Firstly, let us stress that by the trace embedding theorem  the quantity $q(\Omega)$ defined in \eqref{q} is strictly positive.  Let $u$ be a positive eigenfunction corresponding to $\lambda(\beta,\Omega)$ and let $h$ be the solution to the following problem
\begin{equation*}
    \begin{cases}
    \Delta h = 0 & \text{in}\,\Omega\\
    h=0 & \text{on}\,\partial \Theta\\
  \ds  h= u  & \text{on}\,\partial \Omega_0.
    \end{cases}
\end{equation*}
Let us set $v:=u-h$, then we immediately get
\begin{equation}
\label{grad}
    \int_{\Omega}|\nabla u|^2\,dx = \int_{\Omega}|\nabla v|^2\,dx+\int_{\Omega}|\nabla h|^2\,dx.
\end{equation}
Denoted by $D(u)=\ds \int_{\Omega}|\nabla u|^2\,dx$, by using the Minkowski inequality, \eqref{q} and \eqref{grad}, we get
\begin{equation*}
    \begin{split}
\sqrt{\ds\int_{\Omega}u^2\,dx}&\le \sqrt{\ds\int_{\Omega}v^2\,dx}+\sqrt{\ds\int_{\Omega}h^2\,dx} \le \sqrt{\frac{1}{\lambda_{DD}(\Omega)}\ds D(u)}+\sqrt{\frac{1}{q(\Omega)}\ds\int_{\partial \Omega_0}u^2\,d\mathcal{H}^{n-1}}.
\end{split}
\end{equation*}
Squaring both sides we have
\begin{equation*}
    \begin{split}
    \int_{\Omega}u^2\,dx\le \frac{D(u)}{\lambda_{DD}(\Omega)}+ \frac{1}{q(\Omega)}\int_{\partial\Omega_0}u^2\,d\mathcal{H}^{n-1}+2\sqrt{\frac{D(u)}{\lambda_{DD}(\Omega)}\frac{1}{q(\Omega)}\int_{\partial \Omega_0}u^2\,d\mathcal{H}^{n-1}}.
    \end{split}
\end{equation*}
Applying the arithmetic-geometric mean inequality and the H\"{o}lder inequality we get
\begin{equation*}
    \begin{split}
     \int_{\Omega}u^2\,dx &\le \frac{D(u)}{\lambda_{DD}(\Omega)}+\frac{1}{q(\Omega)}\int_{\partial \Omega_0}u^2\,d\mathcal{H}^{n-1}+ \frac{D(u)}{\beta q(\Omega)}
     + \frac{\beta}{\lambda_{DD}(\Omega)}\int_{\partial \Omega_0}u^2\,d\mathcal{H}^{n-1}\\
     &= \bigg(\frac{1}{\lambda_{DD}(\Omega)}+\frac{1}{ \beta q(\Omega)}\bigg)\bigg(D(u)+\beta\int_{\partial \Omega_0}u^2\,d\mathcal{H}^{n-1}\bigg),
     \end{split}
\end{equation*}
that gives the conclusion \eqref{Kutt_40i}.
\end{proof}
\begin{remark}
The previous result gives
\begin{equation}
\label{est_lim}
\lim_{\beta \to +\infty 
}\lambda(\beta,\Omega)=\lambda_{{DD}}(\Omega).
\end{equation}
Moreover, taking $w=1$ as test function in \eqref{q}, we get the following estimate
\begin{equation}
    \label{diff-stima}
\frac{1}{\lambda(\beta,\Omega)}-\frac{1}{\lambda_{{DD}}(\Omega)}\le\frac{|\Omega|}{\beta P(\Omega_0)}.
\end{equation}
Hence, if $\Omega_0$ is convex,  denoted by  $d_M:=\max\{d(x,\partial\Theta) : x \in \de \Omega_0\}$, we get
\[
\lim_{d_M\to 0^+
}\lambda(\beta,\Omega)=\lambda_{{DD}}(\Omega).
\]
for any fixed $\beta>0$.

Furthermore, by \eqref{pmi} and \eqref{diff-stima}, we have 
\begin{equation*}
\frac{1}{\lambda(\beta,\Omega)}-\frac{1}{\lambda_{{DD}}(\Omega)}\le\frac{\rho(\Omega_0)}{\beta}.
\end{equation*}
Finally, for any fixed $\beta$, we obtain
\[
\lim_{\rho(\Omega_0) \to 0
}\lambda(\beta,\Omega)=\lambda_{{DD}}(\Omega).
\]
\end{remark}

\subsection{The problem on the spherical shell} 
In this section, we study the first Robin-Dirichlet eigenvalue on the spherical shell $A_{R_1,R_2}$. In this case, problem \eqref{min_pb} reduces to
\begin{equation}
    \label{min_pb_rad}
\lambda(\beta,A_{R_1,R_2})=\min_{\substack{\phi\in H^1_{\partial B_{R_1}}(A_{R_1,R_2})\\ \phi \nequiv 0}}\frac{\displaystyle\int_{A_{R_1,R_2}}|\nabla \phi|^2dx+\beta\int_{\partial B_{R_2}}\phi^2d\mathcal H^{n-1}}{\displaystyle\int_{A_{R_1,R_2}} \phi^2dx}.
\end{equation}
By \Cref{prop_existence},  any  minimizer of \eqref{min_pb_rad} has constant sign and solves the following problem
\begin{equation}
\label{eig_pb_rad}
\begin{cases}
-\Delta v =\lambda(\beta,A_{R_1,R_2})v & \text{ in } A_{R_1,R_2}\\
\partial_\nu v+\beta v=0 & \text{ on }\partial B_{R_2}
\\
v=0 & \text{ on }\partial B_{R_1}.
    \end{cases}
\end{equation}
The eigenfunctions of \eqref{eig_pb_rad} inherit symmetry properties, as proved in the following result.
\begin{proposition}
\label{prop_rad}
Let be $\beta>0$ and $R_2>R_1>0$. Let $v$ be a positive eigenfunction corresponding to $\lambda(\beta, A_{R_1,R_2})$. Then $v$ is  radially symmetric, that is there exists a positive function $\phi(r)$, $r>0$, such that $v(x)=\phi(|x|)$. Moreover, there exists a unique $\bar R\in (R_1,R_2)$ such that $\phi'(\bar R)=0$ and $\phi(r)$ is strictly increasing in $[R_1,\bar R]$ and is strictly decreasing in $[\bar R,R_2]$.
\end{proposition}

\begin{proof}
The radial symmetry follows from the simplicity of the first eigenvalue and the rotational invariance of the problem. Then if $v$ is a positive eigenfunction, we have $v=\phi(|x|)$, where $\phi(r)$ has to solve the following ordinary differential equation:
\[
\begin{cases}
- \frac{1}{r^{n-1}}\left(\phi'(r)r^{n-1}\right)'=\lambda(\beta, A_{R_1,R_2}) \phi(r) & \text{ for } r\in(R_1,R_2)\\
\phi ' (R_2)+\beta \phi(R_2)=0 \\
\phi(R_1)=0.
\end{cases}
\]
Since $\phi(r)>0$ for $r\in (R_1,R_2)$, then we have
\begin{equation}
\label{sign_der}
\left(\phi'(r)r^{n-1}\right)'<0,
\end{equation}
that means $\phi'(r)r^{n-1}$ is a strictly decreasing function. Since $\phi(R_1)=0$ and $\phi(r)>0$ for $r\in (R_1,R_2)$, then 
\begin{equation}
\label{sign_der_1}
\phi'(r)>0\quad r\in [R_1,R_1+\delta],
\end{equation}
for sufficiently small $\delta>0$. Meanwhile, since we know that $\phi ' (R_2)=-\beta \phi(R_2)<0$, then by continuity 
\begin{equation}
\label{sign_der_2}
\phi'(r)<0 \quad r\in [R_2-\delta',R_2],
\end{equation}
for sufficiently small $\delta'>0$.
The conclusion easily follows from \eqref{sign_der}-\eqref{sign_der_1}-\eqref{sign_der_2}.
\end{proof}
\begin{remark}
\label{livelli}
   Let $v(x)=\phi(r)$, with $r=|x|$, be a positive eigenfunction corresponding to $\lambda(\beta, A_{R_1,R_2})$. By \Cref{prop_rad},  we define $v_m=\phi(R_2)$ and $v_M=\phi(\bar R)$. Then we have (see also \Cref{RDff++})
   \begin{itemize}
   \item[(i)]
   The super-level sets of $v$  are the following
   \begin{equation}
\{x \in A_{R_1,R_2} \colon v(x)>t\}=\begin{cases}
A_{r_i(t),R_2} \text{ if } 0\le t<v_m\\
A_{r_i(t),r_o(t)} \text{ if } v_m\le t<v_M,
\end{cases}
   \end{equation}
   \item[(ii)] The sub-level sets  of $v$  are the following
    \begin{equation}
\{x \in A_{R_1,R_2} \colon v(x)<t\}=\begin{cases}
A_{R_1,r_i(t)} \text{ if } 0< t\le v_m\\
A_{R_1,r_i(t)}\cup A_{r_o(t),R_2} \text{ if } v_m< t\le v_M,
\end{cases}
   \end{equation}
   \end{itemize}
    where $r_i(t)$ is such that $R_1\le r_i(t)\le \bar R$ and $\phi(r_i(t))=t,  $
   while $r_o(t)$ is such that $ \bar{R}\le r_o(t)\le R_2$ and $\phi(r_o(t))=t$.
\end{remark}
\begin{figure}
    \centering
\includegraphics[totalheight=0.4\textwidth,scale=.3]{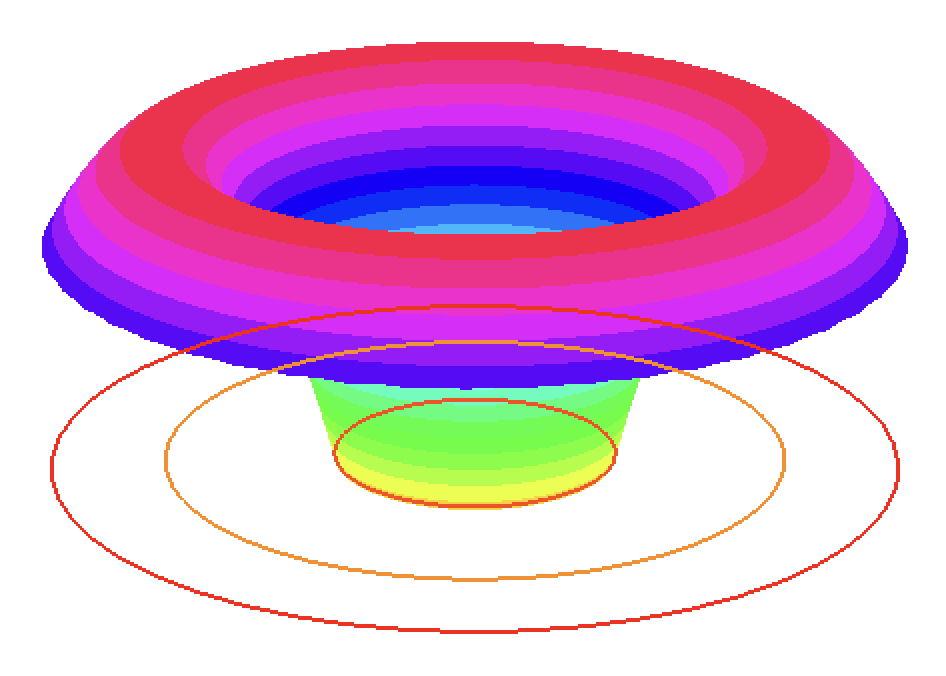}
\caption{The eigenfunction of the Robin-Dirichlet Laplacian on an annulus (the red disks have radii  $R_1$ and $R_2$ respectively).  The orange disk has radius $\bar R$, that corresponds to the maximum level set, that is the set where the monotonicity of the radial function  $\phi$ changes. The figure has been realized by using FreeFem++ \cite{MR3043640}.
}
    \label{RDff++}
\end{figure}

Finally, we prove two 
monotonicity results for the eigenvalue with respect to the inner and the outer radius.
\begin{proposition}
Let $R_2>R_1>0$. Then
    \begin{itemize}
        \item the map $r\in (R_1,R_2)\mapsto \lambda(\beta,A_{R_1,r})$ is strictly decreasing,
        \item the map $r\in (R_1,R_2)\mapsto \lambda(\beta,A_{r,R_2})$ is strictly increasing, 
    \end{itemize}
   for any fixed positive parameter $\beta$.
\end{proposition}
\begin{proof}
Let be $R_1<r<R_2$ and let $\phi\in H^1_{\partial B_{R_1}}(A_{R_1,r})$ be a minimzer of $\lambda(\beta,A_{R_1,r})$. Since by \Cref{prop_rad} $\phi$ is radially symmetric, by a little abuse of notation we indicate $\phi(x)=\phi(|x|)$ and we set
\[
v(x):=\phi\left(R_1+\frac{r-R_1}{R_2-R_1}(|x|-R_1)\right).
\]
Refer to \cite[Lemma 4.1]{bucur2010alternative} for the filled case. By testing \eqref{min_pb_rad} with $v$, we have
\[
\begin{split}
\lambda(\beta, A_{R_1,R_2})&\leq \frac{\displaystyle\int_{A_{R_1,R_2}}|\nabla v|^2dx+\beta\int_{\partial B_{R_2}}v^2d\mathcal H^{n-1}}{\displaystyle\int_{A_{R_1,R_2}} v^2dx}\\
&\leq \frac{\displaystyle \left(\frac{r-R_1}{R_2-R_1}\right)^2 \int_{A_{R_1,r}}|\nabla \phi|^2dx+\beta\left(\frac{r-R_1}{R_2-R_1}\right)\int_{\partial B_{r}}\phi^2d\mathcal H^{n-1}}{\displaystyle \int_{A_{R_1,r}} \phi^2dx}<\lambda(\beta, A_{R_1,r}).
\end{split}
\]
To prove the second claim, let us consider a minimzer $\phi\in H^1_{\partial B_{r}}(A_{r,R_2})$ of $\lambda(\beta, A_{r,R_2})$. Again by a little abuse of notation, we indicate $\phi(x)=\phi(|x|)$ and we set
\[
v(x):=
\begin{cases}
\phi(x) & \text{ in } A_{r,R_2}\\
0       & \text{ in } A_{R_1,r}
\end{cases}
\]
By testing \eqref{min_pb_rad} with $v$, we have
\[
\lambda(\beta, A_{R_1,R_2})\leq  \frac{\displaystyle \int_{A_{r,R_2}}|\nabla v|^2dx+\beta\int_{\partial B_{R_2}}v^2d\mathcal H^{n-1}}{\displaystyle \int_{A_{r,R_2}} v^2dx}
=\lambda(\beta, A_{r,R_2}).
\]
\end{proof}

\section{The first order shape derivative}
\label{sec_4_shape}
In this section, we compute the first domain derivative of $\lambda(\beta, \Omega)$.  For the reader convenience, we recall the Hadamard's formulas (see for instance \cite{sokolowski1992introduction,henrot2006variation}).

Let $V(x)$ a vector field such that $V\in W^{1,\infty}(\R^n;\R^n)$ and, for any $t>0$, let us set 
\[
\begin{split} 
\Omega_0(t) & :=\{ \mathtt x_t:=x+tV(x), \,x\in \Omega_0\},\\    
\Theta(t) & :=\{ \mathtt x_t:=x+tV(x), \,x\in \Theta\},
\end{split}
\]
and $\Omega(t):=\Omega_0(t)\setminus \overline{\Theta(t)}$.
  
Let $f(t,x)$ be such that $f(t,\cdot)\in W^{1,1}(\R^n)$ and differentiable at $t$, then it holds:
\begin{gather}\label{hf1}
\begin{split}
\frac{d}{dt} \int_{\Omega(t)}f(t,x)\, d \mathcal H^{n-1}=&
 \int_{\Omega(t)}\frac{\partial}{\partial t}[f(t,x)]\, d \mathcal H^{n-1} +\\
 & +\int_{+\partial \Omega(t)} f(t,x) \langle V,\nu \rangle\,d \mathcal H^{n-1}.
\end{split} 
\end{gather} 
Moreover, when  the integral is defined on the oriented boundary of $\Omega$, if $f(t,\cdot)\in W^{2,1}(\R^n)$, it holds:
\begin{gather}\label{hf}
\begin{split}
\frac{d}{dt} \int_{+\partial \Omega(t)}f(t,x)\, d \mathcal H^{n-1}=& \int_{+\partial \Omega(t)}\frac{\partial}{\partial t}[f(t,x)] \, d \mathcal H^{n-1} +\\ 
& +(n-1)\int_{+\partial \Omega(t)} H_t(x) \,f(t,x) \langle V,\nu \rangle\,d \mathcal H^{n-1}\\ 
&+\int_{+\partial \Omega(t)}\frac{\partial}{\partial \nu}[f(t,x)]\,\langle V,\nu \rangle \, d \mathcal H^{n-1},
\end{split}
\end{gather} 
where $H_t$ denotes the mean curvature of $\partial \Omega(t)$. For further references, see also \cite[Chap. 5]{henrot2006variation} and \cite{bandle2015second}.

At this stage, we are in position to prove the following.
\begin{theorem}
The first domain derivative of $\lambda(\beta,\Omega)$ in the direction $V$ is
\begin{equation}
\label{derivata}
\begin{split}
\frac{d}{dt}[\lambda(\beta, \Omega(t))]\Big|_{t=0}=&\int_{\de \Omega}\left\{|\nabla u|^2\!\!+\!\! (n-1)\beta Hu^2 - \!(\lambda(\beta, \Omega)+ \!2\beta^2)u^2\right\}
\langle V,\nu \rangle\,d\mathcal H^{n-1}.
\end{split}
\end{equation}
\end{theorem}
\begin{proof}
For any $t\in\R$, we consider the first Robin-Dirichlet eigenvalue of $\Omega(t)$, that is
\begin{equation}\label{min_Omegat}
\lambda(\beta, \Omega(t)):=\min_{\substack{\psi\in H^{1}_{\partial \Theta(t)}(\Omega(t))\\ \psi \not \equiv0}}\left\{{\displaystyle\int_{\Omega(t)}|\nabla \psi|^2dx+\beta\int_{\partial\Omega_0(t)}\psi^2d\mathcal H^{n-1}}\ 
,\ \,\|\psi\|_{L^2(\Omega(t))}=1
\right\}.
\end{equation}

Any  minimizer $u_t\in H^{1}_{\partial \Theta(t)}(\Omega(t))$ of \eqref{min_Omegat} is a solution to the following  perturbed eigenvalue problem
\begin{equation}
\label{eig_Omegat}
\begin{cases}
-\Delta u_t=\lambda(\beta,\Omega(t))u_t & \text{ in }\ \Omega(t)\\
\partial_\nu u_t+\beta u_t=0&\text{ on } \partial\Omega_0(t)\\ 
u_t=0&\text{ on } \partial \Theta (t).
\end{cases}
\end{equation}
Denoted by $u'_t=\frac{\partial}{ \partial t}u_t$, and observing that the perturbed hole $\Theta(t)$ is the zero-level set of the function $u_t$, the inner boundary condition in \eqref{eig_Omegat} implies 
 \begin{equation}
 \label{der_tot_Omegat}
 u_t'= -\partial_\nu u_t (\nu\cdot V) \quad \text{ on } \partial \Theta(t).
 \end{equation} 

By using the boundary conditions in \eqref{eig_Omegat}, \eqref{der_tot_Omegat}, and the divergence theorem, we get
\begin{equation}
    \label{div_t}
\begin{split}
\int_{\Omega(t)}\langle\nabla u_t &,\nabla u_t' \rangle\, d\mathtt x_t=\lambda(\beta, \Omega(t))\int_{\Omega(t)} u_t \,u_t'\,d\mathtt x_t\\
&-\beta \int_{\partial \Omega_0 (t)} u_t\,u_t'\,d\mathcal H^{n-1}-\int_{\partial \Theta (t)}\left(\frac{\partial u_t}{\partial\nu}\right)^2\langle V,\nu \rangle\,d\mathcal H^{n-1}.
\end{split}
\end{equation}
Therefore, by using \eqref{hf1}, \eqref{hf} and \eqref{div_t}, we have:
\begin{equation}
\label{der1}
\begin{split}
\frac{d}{dt}[\lambda(\beta, \Omega(t))]=&2\lambda(\beta, \Omega(t))\int_{\Omega(t)} u_t \,u_t'\,d\mathtt x_t\\
&-\int_{\partial \Theta (t)}\left(\frac{\partial u_t}{\partial\nu}\right)^2\langle V,\nu \rangle\,d\mathcal H^{N-1}\\&+\int_{\de \Omega_0(t)}|\nabla u_t|^2\langle V,\nu \rangle\,d\mathcal H^{n-1}\\
&+(n-1)\beta\int_{\partial \Omega_0 (t)}  H_t \,u_t^2
\langle V,\nu \rangle\,d\mathcal H^{n-1}\\
&+\beta\int_{\partial \Omega_0 (t)}\frac{\partial u_t^2}{\partial\nu}\langle V,\nu \rangle\,d\mathcal H^{n-1}.
\end{split}
\end{equation}
It is straightway seen that the normalization condition in \eqref{min_Omegat} gives
\begin{equation}
\label{normalization_constraint}
2\int_{\Omega(t)} u_t \, u_t'\,d\mathtt x_t+
\int_{\partial \Omega_0 (t)}
u_t^2\,\langle V,\nu \rangle\,d\mathcal H^{n-1}=0.
\end{equation}

Substituting \eqref{normalization_constraint} in \eqref{der1} and recalling that $u_t$ vanishes on $\de \Theta$, we get the conclusion.
\end{proof}

\begin{remark}
Let us observe that  the spherical shell $A_{R_1,R_2}$ such that $|B_{R_1}|=|\Theta|$ and $|B_{R_2}|=|\Omega_0|$  is a stationary domain for the first Robin-Dirichlet eigenvalue.

Indeed, if both the volume of  the outer domain $\Omega_0$ and of the hole $\Theta$ are fixed, it holds:
\begin{equation}
    \label{volume_constraint}
\int_{\partial \Omega_0 (t)} \langle V,\nu \rangle \,d\mathcal H^{n-1}=0\quad \text{and} \quad \int_{\partial\Theta(t)}\langle V,\nu \rangle \,d\mathcal H^{n-1}=0.
\end{equation}

Then, by \Cref{prop_rad}, when the constraint \eqref{volume_constraint} holds, the first derivative in \eqref{derivata} is zero and this suggest that the spherical shell is an optimal shape candidate under both volume constraints of the hole $\Theta $ and $\Omega_0$.
\end{remark}

\section{Proof of the main result}
\label{sec_5_isop}
In this section, we provide the proof of the Theorem  \ref{thm_isop}. 

\begin{proof}[Proof of Theorem \ref{thm_isop}]
Let $v(x)=\phi(\left|x\right|)$ be a positive radial solution to the problem \eqref{min_pb_rad} on the spherical shell $A_{R_1,R_2}$. Let  $\bar R\in (R_1,R_2)$ be the unique point such that $\phi'(\bar R)=0$, whose existence has been proved in \Cref{prop_rad}, let $v_{M}=\phi(\bar R)$ be the maximum of $v$ and $v_m=\phi(R_2)$.

For any $x\in \Omega$, let us denote by $d_o(x)$ the distance of $x$ from $\partial\Omega_0$, that is
\[
d_o(x)=\inf\{\left|x-y\right|, y\in \partial\Omega_0\},
\] 
and by $d_i(x)$ the distance of $x$ from $\Theta$, that is
\[
d_i(x)=\inf\{\left|x-y\right|, y\in \Theta\}.
\] 

Since, $\Omega_0$ and $\Theta$ are convex sets, and $\bar{\Theta} \Subset \Omega_0$,  
we can consider a 
set $\Omega_i \Subset \Omega_0$  such that 
\begin{itemize}
\item $\overline\Theta\subset  \Omega_i$,
\item $|\Omega_i\setminus \overline \Theta|=|A_{R_1, \bar R}|$.
\end{itemize}
At this stage, we denote by $\mathsf M_i=\Omega_i\setminus \overline \Theta$ and by $\mathsf M_o:= \Omega_0 \setminus \overline{\mathsf M_i}$. Therefore, we have (see Figure 2):
\begin{itemize}
\item $\de \mathsf M_o=\de \Omega_0 \cup \de \mathsf M_i$,
\item $|\mathsf M_i|=|A_{R_1, \bar R}|$,
\item $|\mathsf M_o|=|A_{\bar R,R_2}|$.
\end{itemize}

\begin{figure}
\centering
\begin{tikzpicture}[line cap=round,line join=round,>=triangle 45,x=1cm,y=1cm, scale=.47]
\clip(-4,-8) rectangle (16,12);
\draw [rotate around={90:(0,0)},line width=2pt,fill=gray,fill opacity=0.11] (0,0) ellipse (6cm and 2cm);
\draw [rotate around={90:(0,0)},line width=2pt,color=black,fill=purple,fill opacity=0.1] (0,0) ellipse (3.4641016151377544cm and 1.7320508075688772cm);
\draw [rotate around={0:(0,0)},line width=2pt,fill=white,fill opacity=0.98] (0,0) ellipse (1.4142135623730951cm and 0.7071067811865476cm);
\draw [line width=2pt,fill=gray,fill opacity=0.11] (12,0) circle (3.4641016151377544cm);
\draw [line width=2pt,color=black,fill=purple,fill opacity=0.1] (12,0) circle (2.449489742783178cm);
\draw [line width=2pt,fill=white,fill opacity=0.88] (12,0) circle (1cm);
\draw [line width=2pt] (12,0)-- (12.679978343927734,0.7332321950032584);
\draw [line width=2pt] (12,0)-- (11.130352649185248,2.289915606571754);
\draw [line width=2pt] (12,0)-- (8.72934928275161,-1.1414218701963723);
\begin{scriptsize}
\draw[color=blue] (0.7362159722929748,4.564422155210649) node {${\mathsf{M_o}}$};
\draw[color=blue] (0.6733668015612675,1.5790865454545466)  node {${\mathsf{M_i}}$};
\draw[color=blue]
(0.4533947040002919,-0.2174963991130822) node {${\Theta}$};
\draw[color=black] (12.4,-0.30066459423503265) node {${R_1}$};
\draw[color=black] (12.143340460097848,1.6890725942350344) node {${\bar{R}}$};
\draw[color=black] (10.320714508878337,0.27496625277161785) node {${R_2}$};
\end{scriptsize}
\end{tikzpicture}
\label{fig_insiemi}
 \caption{The geometry of the sets $M_i$ and $M_o$ in $\Omega$, of the spherical shells $A_{R_1,\bar R}$ and $A_{\bar R,R_2}$ in $A_{R_1,R_2}$.}
\end{figure}

Keeping in mind Remark \ref{livelli},  we define the following test function 
\[
w(x):=
\begin{cases}
G_o(d_o(x))\quad &\text{if } x\in\mathsf M_o \text{ and } d_o(x)< R_2-\bar R \\
v_M &\text{if } x\in\mathsf M_o \text{ and } d_o(x)\geq  R_2-\bar R\\
G_i(d_i(x)) &\text{if } x\in\mathsf M_i \text{ and } d_i(x)< \bar R- R_1 \\
v_M &\text{if } x\in\mathsf M_i \text{ and } d_i(x)\geq  \bar R- R_1,\\
\end{cases}
\]
where $G_o$ is defined as 
\[
	G_o^{-1}(t)=\int_{v_m}^{t}\dfrac{1}{g_o(\tau)}\;d\tau\qquad v_m< t<v_M,
\]
 with  $g_o(\tau)=|Dv|_{\{v=\tau\}\cap A_{\bar R,R_2}}$ and $G_i$ is defined as 
\[
	G_i^{-1}(t)=\int_{0}^{t}\dfrac{1}{g_i(\tau)}\;d\tau\qquad 0< t<v_M,
\]
 with  $g_i(\tau)=|Dv|_{\{v=\tau\}\cap A_{R_1,\bar R}}$. 
 
We point out that, by Remark \ref{livelli}, there are two radii $r_i(\tau)$ and $r_o(\tau)$ such that  $v(x)=\phi(|x|)$ verifies $ \phi(r_i(\tau))=\phi(r_o(\tau))=\tau$, for  $v_m<\tau<v_M$. This is the reason for which we have to define  two different functions $g_o(\tau)$ and $g_i(\tau)$ in order to apply the web-function technique.

Moreover, we observe that 
\[
v(x)=
\begin{cases}
G_o(R_2-|x|)\quad &\text{if} \  |x|>\bar R\\
G_i(|x|-R_1) &\text{if} \  |x|<\bar R,
\end{cases}
\]
and 
\[
\begin{split}
& G_o(0)=v_m\\
& G_i(0)=0\\
& G_o(R_2-\bar R)=G_i(\bar R-R_1)= v_M.
\end{split}
\]
 Hence, by construction,  $w$ also satisfies the following properties: 
\begin{gather*}
w\in H^{1}(\Omega)\cap C(\bar\Omega)\\
|Dw|_{w=\tau}=|Dv|_{v=\tau}\\
\max_{\bar\Omega} w= v_M.
\end{gather*}

For any $v_m\leq t \leq v_M$, let us denote  
\[ 
\begin{split}
\mathsf E_{t}^o & :=\overline{\Omega_i} \cup \{x\in\mathsf M_o\ \colon\ w(x)>t\} \\ 
\mathsf F_t^o & :=\overline{B_{\bar R}}\cup \{x\in A_{\bar R, R_2} \  \colon\ v(x)>t\}.
\end{split}
\]
By construction, since $\mathsf E^o_{t}$ and $\mathsf F^o_{t}$ are  convex sets,   inequalities \eqref{derivata_composta_super1} and \eqref{iso_W_2} imply
	\begin{equation*}
-\dfrac{d}{dt}P(\mathsf E^o_{t})\geq n(n-1)\dfrac{W_2(\mathsf E^o_{t})}{g(t)}\geq n(n-1)n^{-\frac{n-2}{n-1}}\omega_n^{\frac{1}{n-1}}\dfrac{\left(P(\mathsf E^o_{t})\right)^{\frac{n-2}{n-1}}}{g(t)},
	\end{equation*}
 for $v_m<t<v_M$. Moreover, it holds
	\begin{equation*}
	-\dfrac{d}{dt} P(\mathsf F^o_{t})=n(n-1)n^{-\frac{n-2}{n-1}} \omega_n^{\frac{1}{n-1}} \dfrac{\left(P(\mathsf F^o_{t})\right)^{\frac{n-2}{n-1}}}{g(t)},
	\end{equation*}
for $v_m<t<v_M$. Since $P(\Omega_0)=P(B_{R_2})$ by assumptions, by using a comparison type theorem, we obtain 
	\begin{equation}
 \label{paoli}
P(	\mathsf E^o_{t})\leq P(\mathsf F^o_{t}), 
	\end{equation} 
for $v_m\leq t<v_M$.

By the coarea formula and \eqref{paoli}, we have
\begin{equation}
\label{gradient_estimates_e+_o}
\begin{split}
\int_{\mathsf M_o}|\nabla w(x)|^2\;dx=&
\int_{v_m}^{v_M} g_o(t)\;d\mathcal{H}^{N-1}\left(\partial \mathsf E^o_{t}\cap \mathsf M_o \right)dt\le \\&\le\int_{v_m}^{v_M} g_o(t)d\mathcal{H}^{N-1}(\partial \mathsf F_t^o \cap A_{\bar R, R_2})\;dt=\int_{A_{\bar R, R_2}}|\nabla v(x)|^2\:dx.
\end{split}
\end{equation}
Since, by construction, $w(x)=v_m$ on $\partial \Omega_0$, then 
\begin{equation}\label{termine_bordo_me_o}
\beta\int_{\partial \Omega_0}w^2\:d\mathcal{H}^{N-1}=\beta v^2_m P(\Omega_0)=\beta v^2_m P(B_{R_2})=\beta \int_{\partial B_{R_{2}}}v^2\;d\mathcal{H}^{N-1}.
\end{equation}
Now, let us define $\mu^o(t)=| \mathsf 
 E^o_{t}\cap \mathsf M_o|$ and $\eta^o(t)=|\mathsf F^o_{t}\cap A_{\bar R,R_2}|$. By using again the coarea formula, we obtain
\begin{equation*}
    \begin{split}
\frac{d}{dt}\mu^o(t)&=-\int_{\{ w=t\}\cap\mathsf M_o}\dfrac{1}{|\nabla w(x)|}\;d\mathcal{H}^{n-1}=-\dfrac{\mathcal{H}^{n-1}\left(  \partial \mathsf 
 E_{t}^o\cap\mathsf M_o\right)}{g_o(t)}\\
&\geq -\dfrac{\mathcal{H}^{n-1}( \partial \mathsf F^o_t \cap A_{\bar R,R_2})}{g_o(t)}=-\int_{\{ v=t\}\cap A_{\bar R, R_2}}\dfrac{1}{|\nabla v(x)|}\;d\mathcal{H}^{n-1}=\frac{d}{dt}\eta^o (t),
    \end{split}
\end{equation*}
for any $v_m\le t<v_M$.
This inequality is trivially true also if $0<t<v_m$.
Since $\mu^o(0)=\eta^o(0)=|\mathsf M_o|$, by integrating from $0$ to $t<v_{M}$, we have:
 \begin{equation*}
\mu^o(t)\geq\eta^o(t).
\end{equation*}

Furthermore, we have
\begin{equation}
\label{L_p_estimates_e+_o}
\int_{\mathsf M_o}w^2(x)\;dx=\int_{v_m}^{v_M}2t\mu^o(t)dt \ge \int_{v_m}^{v_M}2t\eta^o(t)\;dt= \int_{A_{\bar R,R_2}}v^2(x) \;dx.
\end{equation}

On the other hand, let us denote  
\[ 
\begin{split}
E_{t}^i & := \overline \Theta \cup  \{x\in\mathsf M_i\ \colon\ w(x)<t\}, \\ 
F_t^i & :=\overline{B_{R_1}}\cup \{x\in A_{\bar R,R_2} \  \colon\ v(x)<t\}.
\end{split}
\]
If we denote
\[
\begin{split}  
\mathsf  E^i_{t}&:=\{x\in\R^n\ \colon\  d_i(x)<G_i^{-1}(t)\}, \\
\mathsf  F^i_{t} &:=  \{x\in \R^n \ \colon\ |x|<R_1+G_i^{-1}(t)\},
\end{split}
\]
then it is worth noting that
\[ 
\begin{split}
E_{t}^i \subseteq \mathsf E^i_{t},\qquad\ F_t^i= \mathsf F_{t}^i .
\end{split}
\]

At this stage, let us set $\delta_i= G_i^{-1}(t)$. We get
\begin{equation}
\label{perimetro_vivo}
\begin{split}
\mathcal H^{n-1}(\partial E_t^i\cap \Omega)&\le P(\mathsf E^i_{t})= P(\Theta+\delta_i B_{1}) \\
& \leq n\sum_{k=0}^{n-1}\binom{n-1}{k}W_{k+1}(\Theta)\delta_i^k\\
& \le n\sum_{k=0}^{n-1}\binom{n-1}{k}W_{k+1}(B_{R_{1}})\delta^k\\
&= P(B_{R_{1}}+\delta_iB_{1})=P(\mathsf F_t^i)= P(F^i_{t}),
\end{split}
\end{equation}
where in the second line we have used the Steiner formula, in the third line we have used the Aleksandrov-Fenchel inequalities and that $W_{n-1}(\Theta)=W_{n-1}(B_{R_1})$, in the fourth line we have used again the Aleksandrov-Fenchel inequalities.

Using the coarea formula and \eqref{perimetro_vivo}, we have
\begin{equation}
\label{gradient_estimates_e+}
\begin{split}
\int_{\mathsf M_i}|\nabla w(x)|^2\;dx=&
\int_{0}^{v_M} g_i(t)\;d\mathcal{H}^{n-1}\left(\partial E_{t}^i\cap\mathsf M_i \right)dt\le \\&\le\int_{v_m}^{v_M} g_i(t)d\mathcal{H}^{n-1}(\partial F_t^i\cap A_{R_1,\bar R})\;dt=\int_{A_{R_1,\bar R}}|\nabla v(x)|^2\:dx.
\end{split}
\end{equation}

Now, let us define $\mu^i(t)=|E^i_{t}\cap\Omega|$ and $\eta^i(t)=| F^i_{t}\cap A_{R_1,\bar R}|$. By using again coarea formula, we obtain
\begin{equation}
    \begin{split}
\frac{d}{dt}\mu^i(t)=&\int_{\{ w=t\}\cap\mathsf M_i}\dfrac{1}{|\nabla w(x)|}\;d\mathcal{H}^{N-1}=\dfrac{\mathcal{H}^{n-1}\left(  \partial E^i_{t}\cap\mathsf M_i\right)}{g_i(t)}\\ 
&\leq \dfrac{\mathcal{H}^{n-1}(\partial F^i_t\cap A_{R_1,\bar R})}{g_i(t)}=\int_{\{ v=t\}\cap A_{R_1,\bar R}}\dfrac{1}{|\nabla v(x)|}\;d\mathcal{H}^{n-1}=\frac{d}{dt}\eta^i(t),
    \end{split}
\end{equation}
for $0\le t<v_M$.
Since $\mu^i(0)=\eta^i(0)=0$, by integrating from $0$ to $t<v_{M}$, we have
 \begin{equation*}\label{magmu}
\mu^i(t)\leq\eta^i(t).
\end{equation*}
On the other hand, we have
\begin{equation}
\label{L_p_estimates_e+}
\int_{\mathsf M_i}w^2(x)\;dx=\int_{0}^{v_M}2t(|\mathsf M_i| -\mu^i(t))dt \ge \int_{0}^{v_M}2t(|A_{R_1,\bar R}|-\eta^i(t))\;dt= \int_{A_{R_1,\bar R}}v^2(x) \;dx.
\end{equation}
Finally, using \eqref{gradient_estimates_e+_o} and \eqref{gradient_estimates_e+}, \eqref{termine_bordo_me_o}, \eqref{L_p_estimates_e+_o} and \eqref{L_p_estimates_e+}, we achieve
\[
\begin{split}
\lambda(\beta, \Omega)&\leq \frac{\int_{\Omega}|\nabla w|^2\;dx+\beta\int_{\partial\Omega_0}w^2\;d\mathcal{H}^{n-1}   }{\int_{\Omega}w^2\;dx} \\ 
& \leq \frac{\int_{A_{R_1,R_2}}|\nabla v|^2\;dx+\beta\int_{\partial B_{R_2}}v^2\;d\mathcal{H}^{n-1}    }{\int_{A_{R_1,R_2}}v^2\;dx}=\lambda (\beta, A_{R_1,R_2}),
\end{split}
\]
and this concludes the proof.
\end{proof}

\section{Conclusions}
In this Section, we collect some open problems.
\begin{enumerate}

\item To investigate the problem \eqref{min_pb_intro} for negative Robin parameter ($\beta<0$).
\vspace{.3cm}
\item To study the problem \eqref{min_pb_intro} when $\beta$ is a suitable  positive function.
\vspace{.3cm}
\item To extend our results to the $p$-Laplacian operator and to anisotropic operators.
\vspace{.3cm}
\item To study the  first eigenvalue of the Laplacian when  Robin boundary conditions are imposed on both components of the boundary.
\end{enumerate}

\section*{Acknowledgments}
This work has been partially supported by the MiUR-PRIN 2017 grant \lq\lq Qualitative and quantitative aspects of nonlinear PDEs\rq\rq, 
by the MiUR-PRIN 2022 grant \lq\lq Geometric-Analytic Methods for PDEs and Applications (GAMPA)\rq\rq, by the MiUR PRIN-PNRR 2022 grant \lq\lq Linear and Nonlinear PDE's: New directions and Applications\rq\rq\ and by GNAMPA of INdAM.


\bibliographystyle{
alpha}
\bibliography{bibliografia}
\end{document}